\documentclass{amsart}
\usepackage[]{latexsym,amssymb,amsmath,amsfonts, amsthm}
\usepackage[all,cmtip,ps]{xy}

\def\A{\mathcal{A}}

\def\C{\mathcal{C}}
\def\T{\mathcal{T}}
\def\F{\mathcal{F}}
\def\X{\mathcal{X}}
\def\Y{\mathcal{Y}}
\def\D{\mathcal{D}}

\def\Db{\mathcal{D}^{b}}
\def\H{\mathcal{H}}

\def\End{\operatorname{End}}
\def\Im{\operatorname{Im}}
\def\Hom{\operatorname{Hom}}
\def\Ext{\operatorname{Ext}}
\def\Tor{\operatorname{Tor}}

\def\pdim{\operatorname{proj\,dim}}

\def\dualita#1#2{\mathrel{
                 \mathop{\vcenter{
                 \offinterlineskip
                 \hbox to 0.6truecm{\rightarrowfill}
                 \hbox to 0.6truecm{\leftarrowfill}}}%
                 \limits_{#2}^{#1}}}
\DeclareMathOperator{\Add}{Add}
\DeclareMathOperator{\add}{add}
\DeclareMathOperator{\Prod}{Prod}

\DeclareMathOperator{\Gen}{Gen}
\DeclareMathOperator{\Pres}{Pres}
\DeclareMathOperator{\Ann}{Ann}

\DeclareMathOperator{\Ker}{Ker}
\DeclareMathOperator{\Rej}{Rej}
\DeclareMathOperator{\cone}{cone}

\DeclareMathOperator{\Coker}{Coker}

\newtheorem{theorem}{Theorem}[section]

\newtheorem{corollary}[theorem]{Corollary}

\newtheorem{definition}[theorem]{Definition}
\newtheorem{example}[theorem]{Example}

\newtheorem{lemma}[theorem]{Lemma}

\newtheorem{proposition}[theorem]{Proposition}
\theoremstyle{remark}
\newtheorem{remark}[theorem]{Remark}

\newcommand*{\rMod}{\textrm{\textup{Mod-}}}

\begin{document}

\title[Abelian categories with a tilting object]{When an abelian category with a tilting object is a module category}
\author{Riccardo Colpi, Francesca Mantese, Alberto Tonolo}
\address[R. Colpi]{ Dip. Matematica Pura ed Applicata, Universit\`a degli studi di Padova, via Trieste 63, I-35121 Padova Italy}
\email{colpi@math.unipd.it}
\address[F. Mantese]{Dipartimento di Informatica, Universit\`a degli Studi di Verona, strada Le Grazie  15, I-37134 Verona - Italy}
\email{francesca.mantese@univr.it}
\address[A. Tonolo]{ Dip. Matematica Pura ed Applicata, Universit\`a degli studi di Padova, via Trieste 63, I-35121 Padova Italy}
\email{tonolo@math.unipd.it}
\thanks{Research supported by grant CPDA071244/07 of Padova University}

\date{\today}
\maketitle

\begin{abstract}
An abelian category with arbitrary coproducts and a small projective generator is equivalent to a module category \cite{Mit}. A tilting object in a abelian category is a natural generalization of   a small projective generator. Moreover, any abelian category with a tilting object admits arbitrary coproducts \cite{CGM}. It naturally arises the question when an abelian category with a tilting object is equivalent to a module category. By \cite{CGM} the problem simplifies in understanding when, given  an associative ring $R$ and  a faithful torsion pair $(\X,\Y)$  in the category of right $R$-modules,  the \emph{heart of the $t$-structure}   $\H(\X,\Y)$ associated to $(\X,\Y)$  is equivalent to  a category of modules. In this paper we give a complete answer to this question,  proving necessary and sufficient condition on $(\X,\Y)$ for $\H(\X,\Y)$  to be equivalent to a module category. We analyze in detail the case  when $R$ is right artinian.
\end{abstract}

\section*{Introduction}

In 1964 Barry Mitchell characterized the module categories as those abelian categories with arbitrary coproducts possessing a small and projective generator \cite{Mit}. At the beginning of the eighties, with the papers of Brenner and Butler, Happel and Ringel, Bongartz and others, the notion of tilting module has been introduced and extensively studied. Tilting modules are small and projective exactly in the subcategory generated by them: they naturally generalize small projective generators.

Tilting theory has been object of further generalizations in the direction of abstract categories, like the case of derived categories \cite{H}, \cite{R}, Grothendieck categories \cite{C} and abelian categories \cite{HRS}, \cite{CF}. In particular in \cite{CF} tilting objects for an arbitrary abelian category are defined. Any abelian category with a tilting object admits arbitrary coproducts \cite{CGM}.

Thus, it naturally arises the question when an abelian category with a tilting object is equivalent to a module category. The aim of this paper is to give necessary and sufficient conditions to guarantee such an equivalence.

In \cite{BBD} Beilinson, Bernstein and Deligne introduced the \emph{heart of a $t$-structure} in a derived category, proving that it is always an abelian category.
In \cite{CGM} Colpi, Gregorio and Mantese showed that an abelian category with a tilting object $V$ is equivalent to the heart $\H(\X,\Y)$ of the $t$-structure in $\D^b(\End V)$ naturally associated with a suitable faithful torsion pair $(\X,\Y)$ in the category of right $\End V$-modules. 
In the light of this result, our question simplifies in understanding when, given an arbitrary associative ring $R$,  the heart of the $t$-structure naturally associated with a faithful torsion pair $(\X,\Y)$ in $\rMod R$ is equivalent to a module category.
This is not always true: for instance the heart associated with the usual torsion pair in the category $\rMod \mathbb{Z}$ of abelian groups is not equivalent to any module category (see Example~\ref{example:abgroups}).

Many papers deal with the problem of understanding when the heart of a $t$-structure is equivalent to a module category in different frameworks (see for example \cite{BeRe}, \cite{HRS}, \cite{AlN}, \cite{HKM}, \cite{CGM}, ...). In all these papers, in different ways, a ``tilting notion'' is always involved.

For example a wide description of the heart associated with a torsion pair is given by Happel, Reiten and Smal$\o$ in \cite{HRS}. In particular they prove, in case of locally finite abelian categories, that if the torsion class is \emph{cogenerating}, i.e., it contains all injective modules, then the heart is equivalent to a module category if and only if the torsion class is generated by a tilting module.

Dually, in our setting, we deal with a faithful torsion pair, that is a torsion pair whose torsion free class contains all projective modules.  Following \cite{CGM} or \cite{HKM} the heart associated with a faithful torsion pair is equivalent to a module category if and only if it is the heart of a $t$-structure generated by a tilting complex. Unfortunately this condition is not easily verifiable. In this paper we want to find an explicit characterization in terms of the torsion pair we start from, as Happel, Reiten and Smal$\o$ did in \cite{HRS} in the case the torsion class is cogenerating.

In our main result Theorem~\ref{thm:main} we give necessary and sufficient conditions on a faithful torsion pair $(\X,\Y)$ in $\rMod R$ for $\H(\X,\Y)$ to be equivalent to a category of modules. In particular if $R$ is artinian, we prove that $\H(\X,\Y)$ is  equivalent to a category of modules if and only the torsion class $\X$ is generated by a finitely presented $R$-module $V$ which is tilting in $\rMod {R/\!\Ann(V)}$ (see Corollary~\ref{cor:artin}).

In the last section we give concrete examples of how our results apply both in the artinian and not artinian cases. Moreover we give a new proof of the fact that a quasi tilted algebra of finite representation type is tilted, originally proved by Happel and Reiten in \cite{HR}.

\section*{Notation}
Let $\C$ be an abelian category and $V$ an object of $\C$. It is possible to associate with $V$ several classes of objects:
\begin{itemize}
\item $\Gen V=\{M\in\C: V^{(\alpha)}\to M \to 0 \text{ is exact in }\C\text{ for a cardinal }\alpha\}$;
\item $\overline{\Gen V}$
 is the closure of $\Gen V$ under subobjects;
\item $\overline{\overline{\Gen V}}$ is the class of objects in $\C$ which admit a finite filtration with consecutive factors in
$ \overline{\Gen V}$;
\item $V^\perp=\Ker\Ext^1_\C(V,-)$.
\end{itemize}

Following Dickson \cite{Di}, a \emph{torsion theory} for $\C$ is a pair $(\T,\F)$ of classes of objects of $\C$ satisfying 
\begin{enumerate}
\item $\T\cap\F=\{0\}$, 
\item $T\rightarrow A\rightarrow 0$ exact and $T\in\T$ imply $A\in\T$; 
\item $0\rightarrow A\rightarrow F$ exact and $F\in\F$ imply $A\in\F$;
\item for each $X\in\C$ there is an exact sequence $0\rightarrow T\rightarrow X\rightarrow F\rightarrow 0$ with $T\in\T$, $F\in\F$.
\end{enumerate}
In such a case $\T$ is a \emph{torsion class} and $\F$ is a \emph{torsion-free class}.

If $R$ is an associative ring with $1\not=0$ and $M$ a right $R$-module, we will denote by $R_M$ the quotient ring $R/\!\Ann_R(M)$.

\section{The heart}

Given any associative ring $R$, let $\Db(R)$ be the bounded derived category of $\rMod R$. For any complex $M^{\bullet}\in \Db(R)$ we denote by $H^{i}(M^{\bullet})$ the $i$-th cohomology  of  $M^{\bullet}$.  If $(\X, \Y)$ is a torsion pair in $\rMod R$, then we denote by $\H(\X, \Y)$ the full subcategory of
$\Db(R)$ defined as
\begin{equation*}
\H(\X, \Y) = \{M^\bullet \in \Db(R) \mid H^{-1}(M^\bullet)\in\Y,\ H^{0}(M^\bullet)\in\X,\ H^{i}(M^\bullet)=0
\ \forall i\neq -1,0\}.
\end{equation*}
$\H(\X, \Y)$ is called the \emph{heart} associated with $(\X, \Y)$; it is the heart of the $t$-structure $(\D^{\leq 0},\D^{\geq 0})$ in $\Db(R)$ where
\[\D^{\leq 0}=\{A^\bullet\in\Db(R): H^i(A^\bullet)=0\text{ for }i>0, H^0(A^\bullet)\in\X\}\text{ and}\]
\[
\D^{\geq 0}=\{A^\bullet\in\Db(R): H^i(A^\bullet)=0\text{ for }i<-1, H^{-1}(A^\bullet)\in\Y\}
\]
In 1982 Beilinson, Bernstein and Deligne \cite{BBD} proved that the heart is an abelian category.
\begin{remark}\label{rem:ext}
Following \cite{BBD}, a sequence $0\to A^\bullet\to B^\bullet\to C^\bullet\to 0$ in $\H(\X, \Y)$ is short exact if and only if $A^\bullet\to B^\bullet\to C^\bullet\to A^\bullet[1]$ is a triangle in $ \Db(R)$. In particular it follows that given $A^\bullet$, $C^\bullet$ in  $\H(\X, \Y)$ the Yoneda $\Ext^i_{\H(\X, \Y)}(C^\bullet, A^\bullet)$ coincides with
$\Hom_{\Db(R)}(C^\bullet, A^\bullet[i])$.
\end{remark}

In the sequel, we shortly denote by $M_1\to M_0$ the complex \[\dots\to  0\to M_1\to M_0\to 0\to \dots\] with zero terms everywhere except in degree $-1$ e $0$.
If $(\X, \Y)$ is a \emph{faithful} torsion pair in $\rMod R$, i.e. $R$ belongs to $\Y$, then each $M^\bullet$ in $\H(\X,\Y)$ is isomorphic to a complex $Y_1\to Y_0$ with terms in $\Y$, which is obtained considering a truncation of a projective resolution of $M^\bullet$. 
In particular we have
\begin{lemma}\label{lemma:representing_objects_in_H}
Let $(\X,\Y)$ be a faithful torsion pair in $\rMod R$ and let $Y_{1},Y_0\in\Y$.
An object $M^\bullet$ of $\H(\X,\Y)$ is isomorphic to the complex $Y_1\stackrel{\phi}{ \to} Y_0$  if and only if
the sequence $0 \to M^\bullet \to Y_1[1] \stackrel{\phi[1]}\to Y_0[1] \to 0$ is exact in $\H(\X,\Y)$.
\end{lemma}
\begin{proof}
By Remark~\ref{rem:ext},  the sequence
\[0 \to M^\bullet \to Y_1[1] \stackrel{\phi[1]}\to Y_0[1] \to 0\]
is exact in $\H(\X,\Y)$ if and only if
\[M^\bullet \to Y_1[1] \stackrel{\phi[1]}\to Y_0[1] \to M^\bullet[1]\]
is a triangle in $\Db(R)$, i.e. if and only if 
\[M^\bullet\cong \cone(\phi[1])[-1]=Y_1\stackrel{\phi}{ \to} Y_0.\]
\end{proof}

In \cite{No} Noohi has given a useful explicit description of morphisms in $\H(\X,\Y)$. Given two objects $M^{\bullet}:=M_1\to M_0$ and $N^{\bullet}:=N_1\to N_0$ in $\H(\X,\Y)$, a morphism between $M^{\bullet}$ and $N^{\bullet}$ is a isomorphism class of commutative diagrams
\[
\xymatrix@-1pc{
M_1\ar[dd]_d\ar[dr]^k&&N_1\ar[dd]^d\ar[dl]_\iota\\
&E\ar[dl]^\sigma\ar[dr]_\rho\\
M_0&&N_0
}\] 
such that the diagonal maps compose to zero and the sequence 
\[0\to N_1\stackrel{\iota}\to E\stackrel{\sigma}\to M_0\to 0\]
is exact. The kernel and cokernel of this morphism are given by the complexes
\[M_1\stackrel k\to A\quad\text{and}\quad E/A \stackrel \rho\to N_0\]
where $A$ is the unique submodule of $E$ sitting between $\Im k$ and $\Ker\rho$ such that $A/\Im k\in\X$ and $\Ker\rho/A\in\Y$.

\section{The problem}

Happel, Reiten and Smal$\o$ in \cite{HRS} have introduced the notion of tilting object in locally finite abelian categories, generalizing that of \emph{classical 1-tilting} module, i.e., a finitely generated tilting module of projective dimension $\leq 1$. 
 Colpi and Fuller in \cite{CF} have further generalized this notion for an arbitrary abelian category:
\begin{definition}\cite[Definition~2.3]{CF}\label{def:tilting}
An object $V$ in an abelian category $\C$ is called \emph{tilting} if:
\begin{enumerate}
\item $\C$ contains arbitrary coproducts of copies of $V$;
\item $V$ is selfsmall (i.e., $\Hom_\C(V,V^{(\alpha)})\cong (\End V)^{(\alpha)}$ for any cardinal $\alpha$);
\item $\Gen V=V^\perp$;
\item $\overline\Gen V=\C$.
\end{enumerate}
\end{definition}
By \cite[Proposition 2.1]{CF} a tilting object has projective dimension $\leq 1$.

Possessing a tilting object is a very tightening condition for an abelian category. In particular it is an AB4 category, i.e. it has arbitrary and exact coproducts (see \cite[Lemma~3.2]{CGM}). On the other hand, an object $V$ in an AB4 category  is tilting if and only if it satisfies conditions (2), (3) in Definition~\ref{def:tilting} (see \cite[Remark 2.2]{CF}). 

A tilting object $V$ in an abelian category $\C$ generates a torsion pair $(\T=\Gen V,\F)$ which is counter equivalent (see \cite{CbF} and \cite{CF}) to a faithful torsion pair $(\X,\Y)$ in $\rMod\!\End V$, called the \emph{tilted torsion pair} of $(\T,\F)$. Precisely the torsion class $\X$ coincides with the image in $\rMod\! \End V$ of the functor $\Ext^1_{\C}(V,-)$ and the torsion-free class $\Y$ coincides with the image in $\rMod\!\End V$ of the functor $\Hom_{\C}(V,-)$.
\begin{example}\label{ex:cuore}
If $(\X, \Y)$ is a faithful torsion pair in $\rMod R$, the complex $R[1]$ is a tilting object in the abelian category $\H(\X,\Y)$. It is $\Gen_{\H(\X,\Y)} R[1]=\Y[1]$ and $(\Y[1],\X)$ is a torsion pair in $\H(\X,\Y)$ naturally counter equivalent to the torsion pair $(\X,\Y)$ in $\rMod\!\End R[1]=\rMod R$.
The counter equivalence between 
$(\Y[1], \X)$ and $(\X,\Y)$ is given by the functors
\[H:=\Hom_{\H}(R[1], -):\H(\X,\Y)\to \rMod R,\text{ and}\]
\[ H':=\Ext_{\H}^1(R[1], -):\H(\X,\Y)\to \rMod R\]
and by their adjoint functors $T$ and $T'$. 
For any complex $M^\bullet$ in $\H(\X,\Y)$ and any module $N$ in $\rMod R$, denoted by $t_\X$ the radical of the torsion pair $(\X,\Y)$, it is
\[H(M^\bullet)=H^{-1}(M^\bullet), \ 
H'(M^\bullet)=H^{0}(M^\bullet), \ 
T(N)=\left(N/t_\X(N)\right)[1], \ 
T'(N)=t_\X(N)[0].\]
\end{example}

The last is much more than an example. Indeed, Colpi, Gregorio and Mantese proved 
\begin{theorem}\cite[Corollary~2.4]{CGM}
Let $\C$ be an abelian category with a tilting object $V$, and $(\T,\F)$ the torsion pair generated by $V$. Then the category $\C$ is equivalent to the heart associated with the tilted torsion pair $(\X,\Y)$ of $(\T,\F)$ in $\rMod\!\End V$.
Moreover $(\End V)[1]$ is the tilting object in $\H(\X,\Y)$ corresponding  to $V$ by the equivalence.
\end{theorem}

In 1964 Barry Mitchell characterized the module categories as those abelian categories with arbitrary coproducts possessing a small and projective generator \cite{Mit}. Since the tilting objects are a natural generalization of small projective generators, it is natural to place the following

\textbf{Problem}: when an abelian category $\A$ with a tilting object $V$ is equivalent to a module category?

By the above quoted result of Colpi, Gregorio and Mantese this problem is equivalent to understand when the heart $\H(\X,\Y)$ associated with a faithful torsion pair $(\X,\Y)$ is equivalent to a category of modules.


\section{Quasi-tilting and tilting modules}

Let $R$ be an associative ring. Applying Definition~\ref{def:tilting} to $\A=\rMod R$ we have that a right $R$-module $V$ is \emph{tilting} if $\Gen V=V^\perp$ or equivalently (see \cite[Proposition~1.3]{CT}) if
\begin{enumerate}
\item[(T1)] there exists a short exact sequence $0\to R_1\to R_0\to V\to 0$ with $R_0$, $R_1$ direct summands of a finite direct sum of copies of $R$;
\item[(T2)] $\Ext^1_R(V,V)=0$;
\item[(T3)] there exists a short exact sequence $0\to R\to V_0\to V_1\to 0$ with $V_0$, $V_1$ direct summands of a finite direct sum of copies of $V$.
\end{enumerate}
Let us emphasize that this notion of ``tilting module'' corresponds in the recent literature to that of ``classical 1-tilting module'' \cite[Definition~5.1.1]{GT}.

In \cite{CDT} the following generalization of the tilting notion has been studied:
\begin{definition}[Definition~2.2, \cite{CDT}]
A right $R$-module $V$ is called \emph{quasi-tilting} if it is finitely generated and
\[\Gen V_{R}=\overline{\Gen} V_{R}\cap V_R^{\perp}\]
\end{definition}
Quasi-tilting modules represent the equivalences between a torsion class and a torsion-free class in categories of modules (see \cite[Theorem~2.6]{CDT}. They are an effective generalization of the tilting notion: in \cite[Proposition~2.3]{CDT} it is proved that a quasi-tilting module $V_R$ is a tilting module if and only if it is faithful and $\Gen V$ is closed under products. 

Given a right module $V$ we will denote by $R_{V}$ the quotient ring $R/\Ann V$.

\begin{proposition}\label{prop:quasiversustilting}
Let $V_R$ be a right $R$-module. 
\begin{enumerate}
\item If $V$ is a quasi-tilting $R$-module and $\Gen V_R$ is closed under products, then 
$V$ is a tilting $R_{V}$-module.
\item If $V$ is a tilting $R_{V}$-module and $\Gen V_R$ is closed under extensions, then $V$ is a quasi-tilting $R$-module.
\end{enumerate}
\end{proposition}
\begin{proof}
First we observe that $\Pres V_{R_{V}}=\Pres V_R$, $\Gen V_{R_{V}}=\Gen V_R$ and $V_{R_{V}}$ is finitely generated if and only if $V_R$ is finitely generated.

1. Since $V_R^{\perp}\cap \rMod R_{V}\subseteq V_{R_{V}}^{\perp}$, by
 \cite[Proposition~2.1.(iii)]{CDT}  $V$ is also a quasi-tilting right $R_{V}$-module. By \cite[Proposition~2.3.(iv)]{CDT} if $\Gen V_R$ is closed under products, then $V$ is a tilting $R_{V}$-module.
 
 2. By \cite[Proposition~2.1.(iii)]{CDT}, we have to prove that $\Gen V_R\subseteq V_R^{\perp}$. Consider a short exact sequence in $\rMod R$
 \[0\to X\to Z\to V\to 0;\]
 since $\Gen V_R$ is closed under extensions, the right $R$-module $Z$ belongs to $\Gen V_R$. Then $0\to X\to Z\to V\to 0$ is also a short exact sequence in $\rMod R_{V}$; since $V$ is a tilting $R_{V}$-module, the sequence splits.
\end{proof}

\section{Necessary conditions}

In this section we will give some necessary conditions for the heart associated with a faithful torsion pair to be equivalent to a whole category of modules.

\begin{lemma}\label{lemma:basic}
Let $(\X,\Y)$ be a faithful torsion pair in $\rMod R$ and assume that there is an equivalence \[\H(\X,\Y)\stackrel{*}{\longleftrightarrow} \rMod S\] between the heart associated with $(\X,\Y)$ and the category of right $S$-modules for a suitable associative ring $S$. 
Then there exists a finitely presented right $R$-module ${V}$ generating $\X$ such that
\begin{enumerate}
\item $V$ is a $R_{V}$-tilting module;
\item the ring $S$ corresponds by the equivalence to a complex $R_1\stackrel f{\to} R_0$, where $
R_1$, $R_0$ are finitely generated projective $R$-modules and $\Coker f=V$.
\end{enumerate}
\end{lemma}
\begin{proof}
Let us denote by $U_S$ the right $S$-module $R[1]^*$; since $R[1]$ is a tilting object in $\H(\X,\Y)$, $U_S$ is a tilting $S$-module. Clearly $R\cong\End R[1]\cong\End U_S$. Therefore by \cite{Mi} also ${}_RU$ is a tilting module, in particular it is finitely presented. Let $(\T=\Gen U_S,\F)$ be the torsion pair generated by $U_S$; composing the natural counter equivalence between 
$(\X,\Y)$ in $\rMod R$ and $(\Y[1],\X)$ in $\H(\X,\Y)$ (see Example~\ref{ex:cuore}) with the equivalence $\H(\X,\Y)\stackrel{*}{\leftrightarrow} \rMod S$, we get a counter equivalence  $\T\leftrightarrow\Y$ and $\F\leftrightarrow\X$ induced by ${}_RU_S$ via the $\Hom_S(U_S,-)$ and $-\otimes_R U$ functors and their first derived functors $\Ext^1_S(U,-)$ and $\Tor^R_1(-,U)$.
It is $\Hom_S(U,M)=H^{-1}(M^*)$ and $N\otimes_R U=((N/t_\X(N))[1])^*$ (see Example~\ref{ex:cuore}).

(1) By \cite[Theorems~2.6, 3.4]{CDT}, $V_R := \Ext^1_S({}_RU_S,S)$ is a quasi tilting module which generates $\X$. Since $\X=\Ker -\otimes{}_RU$ and $_RU$ is finitely presented, it is closed under products; by Proposition~\ref{} $V$ is a tilting $R_{V}$-module by Proposition~\ref{prop:quasiversustilting}.

(2) By property (T3) of tilting modules, there exists a short exact sequence
\[0\to S\to U_1\to U_0\to 0\]
with $U_1,U_0$ direct summands of a finite direct sum of copies of $U_S$. Applying $\Hom_S(U,-)$ we get
\[0\to\Hom_S(U,S)\to R_1\to R_0\to {V}\to 0\]
with the $R_i$'s finitely generated projective $R$-modules. 
In particular $V_R$ is finitely presented.

Since $\Hom_S(U,S)$ belongs to $\Y$ and ${V}$ belongs to $\X$, the complex $R_1\to R_0$ belongs to the heart $\H(\X,\Y)$. Let us see that $S^*$ is isomorphic to the complex $R_1\to R_0$.
Since $U_i\cong R_i\otimes_R U=R_i[1]^*$, we have $U_i^*\cong R_i[1]$. Applying the equivalence $*$ to $0\to S\to U_1\to U_0\to 0$ we get
\[0\to S^*\to R_1[1]\to R_0[1]\to 0.\]
By Lemma~\ref{lemma:representing_objects_in_H}, $S^*$ is isomorphic to the complex $R_1\to R_0$.
\end{proof}

In order to understand when the heart of a faithful torsion pair $(\X,\Y)$ in $\rMod R$ is equivalent to a module category, by Lemma~\ref{lemma:basic} (1), we can assume that $\X=\Gen V_R$ where $V$ is a finitely presented $R$-module and a tilting $R_{V}$-module.

\section{Necessary and sufficient conditions for $\H(\X,\Y)$ to be a module category}

Let us assume $(\X,\Y)$ to be a faithful torsion pair in $\rMod R$ with $\X=\Gen V$ where $V_R$ is finitely presented and a tilting $R_{V}$-module. Since the heart $\H(\X,\Y)$ admits arbitrary coproducts \cite[Lemma~3.1]{CGM}, by \cite[Theorem~3.1]{Mit} it is equivalent to a module category if and only if it has a \emph{small projective generator}. By Lemma~\ref{lemma:basic} (2), we can look for a small projective generator among 
the finitely generated projective presentations of ${V}$ in $\rMod R$.

In the sequel, therefore, we will investigate necessary and sufficient conditions on a finitely generated projective presentation $R_1\stackrel{f}{\to} R_0 \to V_R \to 0$ of the tilting $R_{V}$-module $V$ which generates $\X$ for the complex $R_1\stackrel{f}{\to} R_0$ to be a small projective generator of $\H(\X,\Y)$.

\subsection{The smallness of $R_1\stackrel{f}{\to} R_0$}

It is easy to verify that :
\begin{lemma}\label{lemma:S_is_small}
Let $R_1 \to R_0$ be in $\H(\X,\Y)$, with $R_0, R_1$ finitely generated projectives 
in $\rMod R$. Then $R_1 \to R_0$ is small.
\end{lemma}
\begin{proof}
In \cite[Lemma~4.1]{CF} it is shown that $\H(\X,\Y)$ is closed under coproducts in $\Db(R)$ and that coproducts are defined componentwise. Let us prove that any morphism
\[\phi: (R_1 \to R_0) \to \coprod_{\lambda\in\Lambda} M^\bullet_\lambda\]
in $\H(\X\Y)$ factorizes through a finite coproduct $\coprod_{\lambda\in\Lambda_0} M^\bullet_\lambda$ for a finite subset $\Lambda_0$
of $\Lambda$.
Since $R_0, R_1$ are projective $R$-modules, $\phi$ is a morphism in the homotopic category. We conclude since $R_0, R_1$ are finitely generated modules.
\end{proof}

\subsection{When $R_1\stackrel{f}{\to} R_0$ is projective}

We give now necessary and sufficient conditions for the complex $P^\bullet:=R_1\stackrel{f}{\to} R_0$ to be a projective object in the heart $\H(\X,\Y)$. The result is valid for any complex $M_1\to M_0$ in the heart with projective terms.

\begin{proposition}\label{proposition:projectives_in_H}
The complex $P^\bullet:=R_1\stackrel{f}{\to} R_0$ is projective in $\H(\X,\Y)$ if and only if  
$\Hom_{\Db(R)}(P^\bullet,V [1])=0$,
i.e. for any map $\varphi: R_1 \to V$ there exists a map $\psi: R_0 \to V$ such that 
$\varphi=\psi f$.
\end{proposition}
\begin{proof}
If $P^\bullet$ is projective in $\H(\X,\Y)$, then we have by Remark~\ref{rem:ext}
\[\Hom_{\Db(R)}(P^\bullet,V [1])=\Ext_{\H(\X,\Y)}^1(P^\bullet, V) =0.
\]

Conversely, assume $\Hom_{\Db(R)}(P^\bullet,V[1])=0$. Since $(\Y[1],\X)$ is a torsion pair in $\H(\X,\Y)$, in order to prove that $P^\bullet$ is projective it is enough to check that $\Ext_{\H(\X,\Y)}^1(P^\bullet,\Y[1]) = 0 = \Ext_{\H(\X,\Y)}^1(P^\bullet,\X)$. 
By Lemma~\ref{lemma:representing_objects_in_H}\[
0 \to P^\bullet \to R_1[1]\to R_0[1]\to 0
\]
is an exact sequence in $\H(\X,\Y)$. 

The objects $R_1[1]$ and $R_0[1]$ belong to $\Add (R[1])$ and hence they have projective dimension $\leq 1$; hence first we have $\pdim(P^\bullet)\leq 1$. 

Since $R[1]$ is tilting and $\Gen_{\H(\X,\Y)} R[1]=\Y[1]$, by Definition~\ref{def:tilting} it is $\Add (R[1])\subseteq \Ker\Ext^1(-,\Y[1])$; therefore we get $\Ext^1(P^\bullet,\Y[1])=0$.

Let now $X$ be a right $R$-module in $\X=\Gen V$; consider an epimorphism ${V}^{(\alpha)}\to X\to 0$. It is also an epimorphism between stalk complexes in $\H(\X,\Y)$. Since
\[
\Ext_{\H(\X,\Y)}^1(P^\bullet,{V}^{(\alpha)})=\Hom_{\Db(R)}(P^\bullet,V[1]^{(\alpha)})\subseteq \]
\[\subseteq\Hom_{\Db(R)}(P^\bullet,V[1]^{\alpha})=\Hom_{\Db(R)}(P^\bullet,V[1])^{\alpha}=0,\]
and $\pdim(P^\bullet)\leq 1$, we conclude that $\Ext_{\H(\X,\Y)}^1(P^\bullet, X)=0$.
\end{proof}

\begin{remark}\label{rem:piccolo}
By Proposition~\ref{proposition:projectives_in_H} the complex $R_1\stackrel{f}{\to}R_0$ is projective in $\H(\X,\Y)$ if and only if
\[\phi(\Ker f)=0\quad\text{for each }\phi\in\Hom(R_1,V).\]
Since $R_1\leq^{\oplus} R^m$ for a suitable $m\in\mathbb N$, and 
\[\bigcap_{\phi\in\Hom(R^m,V)}\Ker\phi=(\Ann_R V)^m\]
 we have that $R_1\stackrel{f}{\to}R_0$ is projective if and only if $\Ker f$ as submodule of $R^m$ is contained in $(\Ann_RV)^m$. This condition suggests to choose a presentation $R_1\stackrel{f}{\to} R_0$ of $V$ with $\Ker f$ \emph{as small as possible}.
\end{remark}

The next result goes in the same direction of the above remark. But first let us recall the following useful classical homological result.

\begin{lemma}\cite[Lemma B.1]{JL}
\label{lemma:silvana}
Consider the following diagram in $\rMod R$ with exact rows
\[\xymatrix{
&C\ar[r]^f\ar[d]_h&C'\ar[r]\ar@{.>}[dl]^p\ar[d]&C''\ar@{.>}[dl]^q\ar[r]\ar[d]_{\ell}&0\\
0\ar[r]&L\ar[r]&M\ar[r]_g&N
}
\]
There exists $q:C''\to M$ such that $g\circ q=\ell$ if and only if there exists $p:C'\to L$ such that $p\circ f=h$.
\end{lemma}

\begin{proposition}\label{prop:projectivecover}
If the module $R_1$  in the complex $P^\bullet:=R_1  \overset{f}\to R_0$ is a projective cover of $\Im f$, then $P^\bullet$ is a projective object in $\H(\X,\Y)$.
\end{proposition}
\begin{proof}
Let $\phi$ be a morpism in $\Hom_{\Db(R)}(P^\bullet,{V}[1])$ and denote by $\Omega$ the kernel of $f:R_1\to R_0$. Consider the following diagram in $\rMod R$ with exact rows 
\[
\xymatrix{
0\ar[r]&\Omega\ar[d]_{\hat\phi}\ar[r]^j&R_1\ar[d]_\phi\ar[r]^f&R_0\ar[d]\ar[r]&{V}\ar@{=}[d]\ar[r]&0\\
0\ar[r]&\phi(\Omega)\ar[r]&{V}\ar[r]^g&Q\ar[r]&{V}\ar[r]&0}
\]
where $Q$ is the pushout of the maps $f$ and $\phi$.
Since $\Im g$ belongs to $\Gen {V}\subseteq {V}^{\perp}$ (see Proposition~\ref{prop:quasiversustilting}.(2)), the sequence
\[0\to\Im g\to Q\to {V}\to 0
\]
splits. Therefore by Lemma~\ref{lemma:silvana} and the projectivity of $R_0$ we get
\[
\xymatrix{
0\ar[r]&\Omega\ar[d]_{\hat\phi}\ar[r]^j&R_1\ar[d]_\phi\ar[rr]^f&&R_0\ar[d]\ar[r]\ar@{-->}[dll]_k\ar@{.>}[ddl]&{V}\ar@{=}[d]\ar@{.>}[dl]\ar[r]&0\\
0\ar[r]&\phi(\Omega)\ar[r]&{V}\ar@{->>}[rd]_{\hat g}\ar[rr]^g&&Q\ar[r]&{V}\ar[r]&0\\
{}&&&\Im g\ar@{_(->}[ur]
}
\]
such that $\hat g\circ\phi=\hat g\circ k\circ f$. Let us denote by $\hat k$ the restriction of $k$ to $\Im f$; applying again Lemma~\ref{lemma:silvana} we have the following diagram
\[
\xymatrix{
0\ar[r]&\Omega\ar[d]_{\hat\phi}\ar[r]^j&R_1\ar[d]_{\phi}\ar@{.>}[dl]_{\theta}\ar[r]^{\hat f}&\Im f\ar[d]\ar[r]\ar[dl]_{\hat k}&0\\
0\ar[r]&\phi(\Omega)\ar[r]&{V}\ar[r]^{\hat g}&\Im g\ar[r]&0
}
\]
with $\theta\circ j=\hat\phi$. Therefore, since $\Im\theta=\Im\hat\phi$, we have $R_1=\Im j+\Ker\theta$;  since $\Im j$ is superfluous in $R_1$, we get $\Ker\theta=R_1$, i.e. $\hat\phi=\theta\circ j=0$. Thus the map $\hat g$ is a monomorphism and from $\hat g\circ\phi=\hat g\circ k\circ f$ we get $\phi=k\circ f$. We conclude by Proposition~\ref{proposition:projectives_in_H}.
\end{proof}

\subsection{When $R_1\stackrel{f}{\to} R_0$ is a generator}

Assume $R_1\stackrel{f}{\to} R_0\to {V}\to 0$ is a projective presentation of ${V}$ such that the complex $P^\bullet=R_1\to R_0$ is a projective object in the heart $\H(\X,\Y)$. 
Continue to denote by $\Omega$ the kernel of $f$.

In this section we give necessary and sufficient conditions for  $P^\bullet$ to be a generator in $\H(\X,\Y)$. Let us start with some preliminary lemmas.
\begin{lemma}\label{lemma:sigmaM}
Let $M$ be a right $R$-module and assume $\Gen M=\Pres M$ and closed under extensions. Consider a module $L$ in $\overline{\Gen} M$; then
\begin{enumerate}
\item there exists a short exact sequence
\[0\to L\to X\to M^{(\alpha)}\to 0\]
with $X\in\Gen M$;
\item any extension of $M^{(\alpha)}$ by $L$ belongs to $\overline{\Gen} M$.
\end{enumerate}
\end{lemma}
\begin{proof}
There exists a short exact sequence
\[0\to L\stackrel{\iota}\to X_1\to X_2\to 0\]
with $X_1, X_2$ belonging to $\Gen M$. 

(1) The module $X_2$ is an homomorphic image of $M^{(\alpha)}$ for a suitable cardinal $\alpha$.
Then we have the following diagram with exact rows and columns
\[
\xymatrix{
&&0&0\\
0\ar[r]&L\ar[r]&X_1\ar[u]\ar[r]^p&X_2\ar[u]\ar[r]&0\\
0\ar[r]&L\ar@{=}[u]\ar[r]&X\ar[u]\ar[r]&M^{(\alpha)}\ar[u]^{\pi}\ar[r]&0\\
&&M^{(\beta)}\ar[u]\ar@{=}[r]&M^{(\beta)}\ar[u]
}
\]
where $X$ is the pullback of the maps $p$ and $\pi$. Clearly $X$ results to be an extension of modules in $\Gen M$, and therefore it belongs to $\Gen M$.

(2) Let $Y$ be any extension of $M^{(\alpha)}$ by $L$. We have the following diagram with exact rows
\[\xymatrix{
0\ar[r]&L\ar@{^(->}[d]_\iota\ar[r]^\psi& Y\ar@{^(->}[d]\ar[r]&M^{(\alpha)}\ar@{=}[d]\ar[r]&0\\
0\ar[r]&X_1\ar[r]&Q\ar[r]&M^{(\alpha)}\ar[r]&0
}
\]
where $Q$ is the pushout of $\iota$ and $\psi$. Since $\Gen M$ is closed under extensions, we have that $Q$ belongs to $\Gen M$ and hence $Y$ is subgenerated by $M$.
\end{proof}

Let us denote by $\overline{\overline{\Gen}} {V}$ the class of  right $R$-modules $M$ which admit a finite filtration
\[M=M_0\geq M_1\geq ...\geq M_k=0,\quad k\in\mathbb N\]
with $M_i/M_{i+1}\in\overline\Gen {V}$ for $i=0,..., k-1$.
\begin{lemma}\label{lemma:chiusoquozienti}
The class $\overline{\overline{\Gen}} {V}$ is closed under submodules, quotients and extensions.
\end{lemma}
\begin{proof}
Clearly $\overline{\overline{\Gen}} {V}$ is closed under extensions. Let us see that $\overline{\overline{\Gen}} {V}$ is closed under submodules and quotients. Let $M$ be in $\overline{\overline{\Gen}} {V}$; consider the filtration
\[M=M_0\geq M_1\geq ...\geq M_k=0,\quad k\in\mathbb N\]
with $M_i/M_{i+1}\in\overline\Gen {V}$ for $i=0,..., k-1$.
We prove the closure by induction on the length $k$ of the filtration. If $k=1$, then $M$ belongs to $\overline\Gen {V}$ and the latter is closed under submodules and quotients. Let $k>1$ and $L\leq M$. Consider the diagram with exact rows
\[\xymatrix{0\ar[r]&M_1\ar[r]^\varepsilon&M\ar[r]& M/M_1\ar[r]&0\\
0\ar[r]&L\cap M_1\ar[u]\ar[r]&L\ar@{^(->}[u]^\iota\ar[r]&(L+M_1)/M_1\ar@{=}[u]\ar[r]&0
}
\]
By the inductive hypothesis, since $M_1$ has a filtration of length $k-1$ and $M/M_1$ belongs to $\overline{\Gen}V$, the modules $L\cap M_1$ and $(L+M_1)/M_1$ belong to $\overline{\overline{\Gen}} {V}$. Then we conclude that $L$ belongs to $\overline{\overline{\Gen}} {V}$.
Consider now the diagram with exact rows
\[\xymatrix{0\ar[r]&M_{k-1}\ar@{=}[d]\ar[r]&M\ar[r]^q\ar@{->>}[d]^p& M/M_{k-1}\ar[r]\ar[d]&0\\
&M_{k-1}\ar[r]&M/L\ar[r]&N\ar[r]&0
}
\]
where $N$ is the pushout of $p$ and $q$. Since
$M/M_{k-1}\in\overline{\overline{\Gen}} {V}$ has a filtration of length $k-1$, by induction $N$ belongs to $\overline{\overline{\Gen}} {V}$; then $M/L$ is an extension of $N$ by a homomorphic image of $M_{k-1}\in \overline{\Gen} V$ and hence it belongs to  $\overline{\overline{\Gen}} {V}$.
\end{proof}

We continue to denote by $T$, as in the Example~\ref{ex:cuore}, the functor $\rMod R\to\H(\X,\Y)$ which associates with a module $N_R$ the complex $N/t_\X(N) [1]$ where $t_\X$ is the radical of the torsion pair $(\X,\Y)$.

\begin{lemma}\label{lemma:tecnico}\label{lemma:nuovo}
Let $L\in\Y\cap\overline{\overline{\Gen}} {V}$. Then $T(L)=L[1]$ belongs to $\Gen P^{\bullet}$.
\end{lemma}
\begin{proof}
The module $L$ has a finite filtration
\[L=L_0\geq L_1\geq ...\geq L_k=0,\quad k\in\mathbb N\]
with $L_i/L_{i+1}\in\overline\Gen {V}$ for $i=0,..., k-1$. We prove the claim by induction on the length $k$ of the filtration of $L$. 

$k=1$:
Since $V$ is a tilting $R_{V}$-module, then  $\Gen V=\Pres V$ and the latter is closed under extensions. Therefore we can apply Lemma~\ref{lemma:sigmaM}.(1) to get a short exact sequence
\[0\to N\to X\to {V}^{(\alpha)}\to 0\]
with $X\in\Gen {V}$. Applying the functor $T$ we get $T' {V}^{(\alpha)}\to TN\to TX=0$. It is easy to see that $T' ({V}^{(\alpha)})$ is generated by $P^{\bullet}$: indeed by \cite[Corollary~3.3]{No} the following is an epimorphism on the heart $\H(\X,\Y)$:
\[\xymatrix@-1pc{
&R_1^{(\beta)}\ar[dd]&&0\ar[dd]\\
{P^{\bullet}}^{(\beta)}\ar@{=}[r]&{\phantom{AA}}\ar[rr]&&{\phantom{AA}}\\
&R_0^{(\beta)}&&{V}^{(\beta)}
}\]
Therefore also $TN$ belongs to $\Gen P^{\bullet}$.

$k>1$: We have the following diagram with exact rows and columns
\[\xymatrix{
&&0\ar[d]&0\ar[d]\\
0\ar[r]&L_1\ar@{=}[d]\ar[r]&Q\ar[d]\ar[r]&t(L/L_1)\ar[d]^i\ar[r]&0\\
0\ar[r]&L_1\ar[r]&L\ar@{->>}[d]\ar[r]^p&L/L_1\ar@{->>}[d]\ar[r]&0\\
&&K\ar@{=}[r]&K
}\]
where $t(L/L_1)$ is the torsion part of $L/L_1$ and $Q$ is the pullback of  $i$ and $p$. Since $L_1$ has a filtration of length $k-1$,  by inductive hypothesis  $TL_1\in \Gen P^{\bullet}$. Therefore applying $T$ to the first row of the above diagram we get that $TQ\in\Gen P^{\bullet}$.
Now  $K$ is the torsion free part of $L/L_1$ and therefore it belongs to $\Y$; since $L/L_1\in\overline\Gen {V}$, the module $K$ belongs also to $\overline\Gen {V}$. By Lemma~\ref{lemma:tecnico} we have $TK\in\Gen P^{\bullet}$. Since $Q,L, K$ belong to $\Y$ we have the following short exact sequence in $\H(\X,\Y)$
\[0\to TQ\to TL\to TK\to 0\]
By a standard argument, since $P^{\bullet}$ is projective, $\Gen P^{\bullet}$ is closed under extensions and so $TL$ belongs to $\Gen P^{\bullet}$.
\end{proof}

\begin{proposition}\label{prop:genbarbar}
The projective object $P^\bullet:=R_1\stackrel{f}{\to} R_0$  is a generator of $\H(\X,\Y)$ if and only if
there exists a cardinal $\alpha$ and a morphism $g\in\Hom_R(R_1^{(\alpha)},R)$ such
that the cokernel of the restriction $g_{|\Omega^{(\alpha)}}$ belongs to $\overline{\overline{\Gen}} {V}$.
\end{proposition}
\begin{proof}
Let assume that $P^\bullet$ is a projective generator. Then there is an exact sequence in $\H(\X,\Y)$ of the form 
\[(*)\qquad
0 \to {K}^\bullet \to {P^\bullet}^{(\alpha)}\overset{\phi}\to R[1] \to 0
\]
Here $\phi$ is a complex map from ${P^\bullet}^{(\alpha)} = R_1^{(\alpha)} \overset{f^{(\alpha)}}\to R_0^{(\alpha)}$ to $R[1]$. Clearly only the component on degree -1  of $\phi$ is different from 0: it is the wanted morphism $g$ in $\Hom_R(R_1^{(\alpha)},R)$. 
Indeed, the exact sequence $(*)$ induces the long exact sequences of homologies
\[
H^{-1}{P^\bullet}^{(\alpha)} = \Omega^{(\alpha)} \overset{H^{-1}\phi}\to H^{-1} R[1] = R \to H^0 K^\bullet
\]
Since $H^0 K^\bullet$ belongs to $\X$, and $H^{-1}\phi$ is the restriction of $g$ at $\Omega^{(\alpha)}$, we get that the cokernel of the restriction $g_{|\Omega^{(\alpha)}}$ belongs to $\overline{\Gen} {V}$.

Conversely, assume the existence of the cardinal $\alpha$ and the map $g$ in $\Hom_R(R_1^{(\alpha)},R)$ such that the cokernel of the restriction $g_{|\Omega^{(\alpha)}}$ belongs to $\overline{\overline{\Gen}} {V}$. First notice that in order to prove that $P^{\bullet}$ is a generator it is enough to prove that  $P^{\bullet}$ generates  the tilting object $R[1]$. Indeed, by construction, there is an epimorphism in the heart
$P^{\bullet}\to V$ and hence $P^{\bullet}$ generates the torsion free class $\X=\Gen V$; if $P^{\bullet}$
generates also $R[1]$, then it generates the torsion class $\Y[1]$ and so, being projective, also the extensions of $\X$ by $\Y[1]$, i.e.  the whole category $\H(\X, \Y)$.

Let us consider the following map in the heart:
\[
\xymatrix@-1pc{
&&&R_1^{(\alpha)}\ar[dd]_{f^{(\alpha)}}\ar[r]^g&R\ar[dd]\\
P^{\bullet(\alpha)}\ar[r]^\mu &R[1]&:={}\\
&&&R_0^{(\alpha)}\ar[r]&0
}\]
Let us denote by $h$ the map 
\[R_1^{(\alpha)}\stackrel{\left(g, -f^{(\alpha)}\right)}\longrightarrow R\oplus R_0^{(\alpha)}.\]
Then we have the following diagram with exact rows and columns in $\rMod R$:
\[
\xymatrix{
&0\ar[d]&0\ar[d]&0\ar[d]\\
0\ar[r]&g(\Ker f^{(\alpha)})=g(\Omega^{(\alpha)})\ar[d]\ar[r]&R\ar[d]\ar[r]&R/g(\Omega^{(\alpha)})\ar[d]\ar[r]&0\\
0\ar[r]&\Im h\ar[r]\ar@{->>}[d]&R\oplus R_0^{(\alpha)}\ar@{->>}[d]\ar[r]&R\oplus R_0^{(\alpha)}/\Im h\ar[r]\ar@{->>}[d]&0\\
0\ar[r]&\Im f^{(\alpha)}\ar[r]&R_0^{(\alpha)}\ar[r]&{{V}}^{(\alpha)}\ar[r]&0
}
\]
By hypothesis $R/g(\Omega^{(\alpha)})$ belongs to $\overline{\overline{\Gen}} {V}$; therefore also $R\oplus R_0^{(\alpha)}/\Im h$ belongs to $\overline{\overline{\Gen}} {V}$ by Lemma~\ref{lemma:chiusoquozienti}.
Let $A/\Im h$ with $\Im h\leq A\leq R\oplus R_0^{(\alpha)}$ the torsion part of $R\oplus R_0^{(\alpha)}/\Im h$; then $R\oplus R_0^{(\alpha)}/A$ belongs to $\Y$. By the following diagram and Lemma~\ref{lemma:chiusoquozienti} it belongs also to $\overline{\overline{\Gen}} {V}$:
\[\xymatrix{
0\ar[r]&R/g(\Omega^{(\alpha)})\ar@{=}[d]\ar[r]&R\oplus R_0^{(\alpha)}/\Im h\ar[r]^-p\ar@{->>}[d]^q&{{V}}^{(\alpha)}\ar[r]\ar@{->>}[d]&0\\
&R/g(\Omega^{(\alpha)})\ar[r]&R\oplus R_0^{(\alpha)}/A\ar[r]&X\ar[r]&0}\]
where $X$ is the pushout of $p$ and $q$. 
By Lemma~\ref{lemma:nuovo} $T (R\oplus R_0^{(\alpha)}/A)$ belongs to $\Gen P^\bullet$. Following \cite[section 3.4]{No}, $(R\oplus R_0^{(\alpha)}/A) [1]$ is the cokernel in $\H(\X,\Y)$ of $P^{\bullet(\alpha)}\stackrel{\mu}{\to} R[1]$.
Thus we have the following diagram with exact row in $\H(\X,\Y)$:
\[\xymatrix{P^{\bullet(\alpha)}\ar[r]^{\mu} &R[1]\ar[r] &(R\oplus R_0^{(\alpha)}/A) [1]\ar[r]&0\\
&&P^{\bullet(\beta)}\ar@{->>}[u]\ar@{.>}[ul]^{\lambda}}\]
where $\lambda$ is obtained by the projectivity of $P^{\bullet(\beta)}$.
The map $\lambda\oplus\mu:P^{\bullet(\beta)}\oplus P^{\bullet(\alpha)}\to R[1]$ is an epimorphism in $\H(\X,\Y)$, and hence $R[1]$ is generated by $P^\bullet$.
\end{proof}

\begin{remark}
The previous result suggests to consider a presentation $R_1\stackrel{f}{\to} R_0$ of $V$ with $\Omega=\Ker f$ \emph{as big as possible} to get a map $g:R^{(\alpha)}_1\to R$ with $\Coker g_{|\Omega^{(\alpha)}}$ belonging to $\overline{\overline{\Gen}} {V}$, going in the opposite direction with what we have observed in Remark~\ref{rem:piccolo}.
\end{remark}

Let $\mathcal V:=\{{V}/K: K\leq {V}\}$. For any right $R$-module $M$, consider
\[\Rej_{\mathcal V} M=\bigcap_{f:M\to W,\ W\in\mathcal V}\Ker f.\]
Clearly we have the following \emph{reject chain} of $M$:
\[M\geq \Rej_{\mathcal V} M\geq \Rej^2_{\mathcal V} M:=\Rej_{\mathcal V} (\Rej_{\mathcal V} M)\geq \Rej^3_{\mathcal V} M\geq \Rej^4_{\mathcal V} M\geq ...\]
Each factor $\Rej^i_{\mathcal V} M/\Rej^{i+1}_{\mathcal V} M$ belongs to $\overline\Gen {V}$, for 
\[\Rej^i_{\mathcal V} M/\Rej^{i+1}_{\mathcal V} M\hookrightarrow\prod_{W\in\mathcal V}W^{\Hom(\Rej^i_{\mathcal V} M, W)},\]\[ m+\Rej^{i+1}_{\mathcal V} M\mapsto (\varphi_W(m))_{\varphi_W\in \Hom(\Rej^i_{\mathcal V} M, W), W\in\mathcal V}\]
Therefore for any $i\in\mathbb N$, the module $M/\Rej^i_{\mathcal V} M$ belongs to $\overline{\overline\Gen}{V}$.

\begin{proposition}\label{prop:rejchain}
Assume the chain 
\[R\geq \Rej_{\mathcal V} R\geq \Rej^2_{\mathcal V} R\geq \Rej^3_{\mathcal V} R\geq ...\]
is stationary, and $\Rej^n_{\mathcal V} R=\cap_{i\in\mathbb N}\Rej^{n+i}_{\mathcal V} R$ admits a projective cover $\xi:R_2\to \Rej^n_{\mathcal V} R$. Then $R_1\oplus R_2\stackrel{(f,0)}\to R_0$ is a projective generator of $\H(\X,\Y)$.
\end{proposition}
\begin{proof}
First let us prove that $\Hom_R(R_2,{V})=0$. Given $\zeta\in\Hom_R(R_2,{V})$, consider the following diagram with exact rows:
\[\xymatrix{
R_2\ar[r]^-\xi\ar[d]^\zeta&\Rej^n_{\mathcal V} R\ar[r]\ar[d]^{\pi_2}& 0\\
{V}\ar[r]&Q\ar[r]& 0
}\]
where $Q:={V}\oplus\Rej^n_{\mathcal V} R/<(\zeta(m),\xi(m)):m\in R_2>$ is the pushout of $\xi$ and $\zeta$. Since $\Rej^n_{\mathcal V} R=\Rej^{n+1}_{\mathcal V} R$,  $\pi_2=0$ and hence for each $\ell\in R_2$ it is $\pi_2(\xi(\ell))=0$, i.e. $(0, \xi(\ell))\in <(\zeta(m),\xi(m)):m\in R_2>$. Therefore, for each $\ell\in R_2$ there exists $m_\ell\in \Ker\zeta$ such that $\xi(m_\ell)=\xi(\ell)$. Thus we have $R_2=\Ker \zeta+\Ker\xi$; since $\Ker \xi$ is superfluous, we have $R_2=\Ker \zeta$ and hence $\zeta=0$.
Now since $\Hom_R(R_1\oplus R_2,{V}) =\Hom_R(R_1,{V})$ and $R_1\stackrel{f}\to R_0$ is a projective object in $\H(\X,\Y)$, by Proposition~\ref{proposition:projectives_in_H} also $R_1\oplus R_2\stackrel{(f,0)}\to R_0$ is a projective object in $\H(\X,\Y)$. The cokernel of $R_1\oplus R_2\stackrel{(0,\xi)}\to R$ is $R/\Rej^n_{\mathcal V} R$ which belongs to $\overline{\overline\Gen}{V}$. By Proposition~\ref{prop:genbarbar} we conclude that $R_1\oplus R_2\stackrel{(f,0)}\to R_0$ is also a generator of $\H(\X,\Y)$.
\end{proof}

\section{The main result}

We can now give our results, collecting what we have proved in the previous sections.

\begin{theorem}\label{thm:main}
Let $(\X,\Y)$ be a faithful torsion pair in $\rMod R$. The heart $\H(\X,\Y)$ is equivalent to a module category if and only 
\begin{enumerate}
\item $\X=\Gen V$ where $V$ is a tilting $R_{V}$-module ;
\item $V_R$ admits a presentation 
\[0\to \Omega\hookrightarrow R_1\stackrel f{\to} R_0\to V\to 0\]
with $R_1$ and $R_0$ finitely generated projective modules such that 
\begin{enumerate}
\item any map $R_1\to V$ extends to a map $R_0\to V$,
\item there exists a map $R_1^{(\alpha)}\stackrel g{\to} R$ such that the cokernel of the restriction $g_{|\Omega^{(\alpha)}}$ belongs to $\overline{\overline{\Gen}}V$.
\end{enumerate}
\end{enumerate}
In such a case, $R_1\to R_0$ is a small projective generator of the heart $\H(\X,\Y)$. Denoted by $S$ the endomorphism ring $\End_{\H(\X,\Y)}(R_1\to R_0)$, the heart $\H(\X,\Y)$ is equivalent to $\rMod S$.
\end{theorem}
\begin{proof}
It is an easy consequence of Lemma~\ref{lemma:basic}, Lemma~\ref{lemma:S_is_small}, Proposition~\ref{proposition:projectives_in_H} and Proposition~\ref{prop:genbarbar}.
\end{proof}

If we concentrate our attention to artinian rings the above result assume the following aspect: 
\begin{corollary}\label{cor:artin}
Let $R$ be a right artinian ring and $(\X,\Y)$ be a faithful torsion pair in $\rMod R$. The heart $\H(\X,\Y)$ is equivalent to a module category if and only if 
 $\X=\Gen {V}$ where $V$ is a finitely presented $R$-module and a $R_{V}$-tilting module.
\end{corollary}
\begin{proof}
It follows by Theorem~\ref{thm:main}, Proposition~\ref{prop:projectivecover} and Proposition~\ref{prop:rejchain}.
\end{proof}

For a faithful torsion pair $(\X, \Y)$ in $\rMod R$, the tilting and cotilting notions are strictly related, as the next corollary shows. 
 
\begin{corollary}
Let $(\X, \Y)$ be a faithful torsion pair in $\rMod R$ such that $\X$ is generated by a tilting module. Then $\Y$ is  cogenerated by a cotilting module.
\end{corollary}
\begin{proof}
If $\X=\Gen V$ for a tilting module $V$, since the projective dimension of $V$ is at most one and ${\overline{\Gen}}V=\rMod R$, the assumptions of Theorem~\ref{thm:main} are satisfied. In particular  $\H(\X, \Y)$ is a Grothendieck category and so, by \cite{CGM}, $\Y$ is cogenerated by a cotilting module.
\end{proof}

\begin{remark}
In \cite{CGM} it is shown that the heart of a faithful torsion pair $(\X, \Y)$ is equivalent to a category of modules $\rMod S$ if and only if there exists a tilting complex $E^{\bullet}$ in $\D^b(R)$ such that the heart of the $t$-structure $\H_{E^{\bullet}}$ generated by $E^{\bullet}$ coincides with $\H(\X, \Y)$. In such a case  $E^{\bullet}$ is a small projective generator of $\H(\X, \Y)$ and $\End(E^{\bullet})\cong S$. Where to look for it, how to construct it, in terms of the torsion theory we started from, is not provided. 
Notice that such an $E^{\bullet}$ is quasi-isomorphic to a complex $R_1 {\to} R_0$ satisfying conditions of Theorem~\ref{thm:main}. 

Conversely, following \cite{HKM} (more precisely, applying Theorems~2.10 and 3.8, and Corollary 3.6), we get that a complex $R_1 {\to} R_0$ satisfying the conditions of Theorem~\ref{thm:main} is a tilting complex and the heart of the $t$-structure $\H_{R_1 {\to} R_0}$ generated by $R_1 {\to} R_0$ coincides with $\H(\X, \Y)$.
\end{remark}

\section{Applications}
 The following two examples show how strong  is the property that $\X$ has to be generated by a finitely presented   module for $\H(\X,\Y)$ to be equivalent to a module category.
 
\begin{example}\label{example:abgroups}
\begin{enumerate}
\item
In the category of abelian groups one can consider the torsion pair of torsion and torsion-free abelian groups, and that of divisible and reduced abelian groups.
The heart of both these torsion pairs are not equivalent to a module category. Indeed in both the cases the torsion class is not generated by a finitely generated abelian group.

\item Let $\Lambda$ be a Kronecker algebra.
The closure by direct limits of the class of preinjective modules is a torsion class $\X$; denoted by $\Y$ the corresponding torsion free class, the heart $\H(\X, \Y)$ is not equivalent to a module category.
Indeed, if $\X$ is generated by a finitely presented module ${V}$, then any preinjective module would be a quotient of a direct sum of  finite number of copies of ${V}$.  There is no finitely presented module ${V}$  with this property.
\end{enumerate}
\end{example}

In the following example we prove as for path algebras with relations, both in the artinian and not artinian case, our algorithm permits not only to decide if the heart of a faithful torsion pair is equivalent to a module category over a ring, but in the affirmative case also to construct explicitly the ring itself.

\begin{example}Denote by $k$ an algebraically closed field.
\begin{enumerate}
\item
Let $\Lambda$ be the path $k$-algebra given by the following quiver
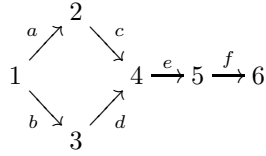
\[\xymatrix@-1pc{
&2\ar[rd]^c\\
1\ar[ru]^a\ar[rd]_b&&4\ar[r]^e&5\ar[r]^f&6\\
&3\ar[ru]_d
}\]
with relations $ca=db=fec=fed=0$.
Let us consider the $\Lambda$-module $V=
\begin{smallmatrix} 1\\ 2\end{smallmatrix}\oplus
\begin{smallmatrix} 1\\ 3\end{smallmatrix}\oplus
\begin{smallmatrix} 1\end{smallmatrix}\oplus
\begin{smallmatrix} 4\end{smallmatrix}$. The subcategory
$\Gen V$ is a torsion class; let us denote by $\Y$ the corresponding torsion free class. The ring $\Lambda_{V}:=\Lambda/\Ann_\Lambda V$ is the path algebra associated with the quiver
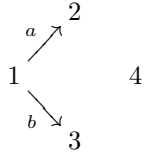
\[\xymatrix@-1pc{
&2\\
1\ar[ru]^a\ar[rd]_b&&4\\
&3
}\]
It is easy to verify that the finitely presented $\Lambda$-module $V$ is a 
tilting $\Lambda_{V}$-module. Therefore, by Corollary~\ref{cor:artin},
$\H(\Gen V, \Y)$ is equivalent to a module category. Since $\Lambda$ is of finite representation type, $\H(\Gen V, \Y)$ is equivalent to a module category over an artin algebra of finite representation type $\Theta$ associated with a suitable quiver with relations. To determine it we have to find the indecomposable projective $\Theta$-modules.  Let us start constructing the small projective generator of  $\H(\Gen V, \Y)$. Consider the following resolution of $V$:
\[0\to \left(\begin{smallmatrix} 4\\ 5\end{smallmatrix}\right)^4\to
\left(\begin{smallmatrix} 2\\ 4\\ 5\end{smallmatrix}\right)^2\oplus
\left(\begin{smallmatrix} 3\\ 4\\ 5\end{smallmatrix}\right)^2\oplus
\begin{smallmatrix} 5\\ 6\end{smallmatrix}\to
\left(\begin{smallmatrix} &1\\ 2&& 3\end{smallmatrix}\right)^3\oplus
\begin{smallmatrix} 4\\ 5 \\ 6\end{smallmatrix}
\to \begin{smallmatrix} 1\\ 2\end{smallmatrix}\oplus
\begin{smallmatrix} 1\\ 3\end{smallmatrix}\oplus
\begin{smallmatrix} 1\end{smallmatrix}\oplus
\begin{smallmatrix} 4\end{smallmatrix}\to 0
\]
Since $\left(\begin{smallmatrix} 4\\ 5\end{smallmatrix}\right)^4$ is superfluous in $\left(\begin{smallmatrix} 2\\ 4\\ 5\end{smallmatrix}\right)^2\oplus
\left(\begin{smallmatrix} 3\\ 4\\ 5\end{smallmatrix}\right)^2\oplus
\begin{smallmatrix} 5\\ 6\end{smallmatrix}$, by Proposition~\ref{prop:projectivecover}, the complex with projective terms
\[\left(\begin{smallmatrix} 2\\ 4\\ 5\end{smallmatrix}\right)^2\oplus
\left(\begin{smallmatrix} 3\\ 4\\ 5\end{smallmatrix}\right)^2\oplus
\begin{smallmatrix} 5\\ 6\end{smallmatrix}\to
\left(\begin{smallmatrix} &1\\ 2&& 3\end{smallmatrix}\right)^3\oplus
\begin{smallmatrix} 4\\ 5 \\ 6\end{smallmatrix}
\]
is a projective object in the heart. 
Let us consider the reject chain of $\Lambda$ with respect to the family of all quotients of $V$:
\[\Lambda=\begin{smallmatrix} &1\\ 2&& 3\end{smallmatrix}\oplus
\begin{smallmatrix} 2\\ 4\\ 5\end{smallmatrix}\oplus
\begin{smallmatrix} 3\\ 4\\ 5\end{smallmatrix}\oplus
\begin{smallmatrix} 4\\ 5\\ 6\end{smallmatrix}\oplus
\begin{smallmatrix} 5\\ 6\end{smallmatrix}\oplus
\begin{smallmatrix} 6\end{smallmatrix}\geq 
\left(\begin{smallmatrix} 4\\ 5\end{smallmatrix}\right)^2\oplus
\left(\begin{smallmatrix} 5\\ 6\end{smallmatrix}\right)^2\oplus
\begin{smallmatrix} 6\end{smallmatrix}\geq 
\left(\begin{smallmatrix} 5\end{smallmatrix}\right)^2\oplus
\left(\begin{smallmatrix} 5\\ 6\end{smallmatrix}\right)^2\oplus
\begin{smallmatrix} 6\end{smallmatrix}=...
\]
The module $\left(\begin{smallmatrix} 5\\ 6\end{smallmatrix}\right)^4\oplus
\begin{smallmatrix} 6\end{smallmatrix}$ is the projective cover of the stationary term of the reject chain.
Therefore by Proposition~\ref{prop:rejchain} the complex
\[\left(\begin{smallmatrix} 2\\ 4\\ 5\end{smallmatrix}\right)^2\oplus
\left(\begin{smallmatrix} 3\\ 4\\ 5\end{smallmatrix}\right)^2\oplus
\begin{smallmatrix} 5\\ 6\end{smallmatrix}\oplus
\left(\begin{smallmatrix} 5\\ 6\end{smallmatrix}\right)^4\oplus
\begin{smallmatrix} 6\end{smallmatrix}
\to
\left(\begin{smallmatrix} &1\\ 2&& 3\end{smallmatrix}\right)^3\oplus
\begin{smallmatrix} 4\\ 5 \\ 6\end{smallmatrix}
\]
is a small projective generator of the heart. Let us decompose it as a direct sum of indecomposable projective complexes:
\[[\begin{smallmatrix} 2\\ 4\\ 5\end{smallmatrix}\to \begin{smallmatrix} &1\\ 2&& 3\end{smallmatrix}]\oplus
[\begin{smallmatrix} 2\\ 4\\ 5\end{smallmatrix}\oplus \begin{smallmatrix} 3\\ 4\\ 5\end{smallmatrix}\to \begin{smallmatrix} &1\\ 2&& 3\end{smallmatrix}]\oplus
[\begin{smallmatrix} 3\\ 4\\ 5\end{smallmatrix}\to \begin{smallmatrix} &1\\ 2&& 3\end{smallmatrix}]\oplus
[\begin{smallmatrix} 5\\ 6\end{smallmatrix}\to \begin{smallmatrix} 4\\ 5\\ 6\end{smallmatrix}]\oplus
[\begin{smallmatrix} 5\\ 6\end{smallmatrix}\to 0]^4\oplus
[\begin{smallmatrix} 6\end{smallmatrix}\to 0]
\]
Forgetting the redundant repetition we get that also 
\[[\begin{smallmatrix} 2\\ 4\\ 5\end{smallmatrix}\to \begin{smallmatrix} &1\\ 2&& 3\end{smallmatrix}]\oplus
[\begin{smallmatrix} 2\\ 4\\ 5\end{smallmatrix}\oplus \begin{smallmatrix} 3\\ 4\\ 5\end{smallmatrix}\to \begin{smallmatrix} &1\\ 2&& 3\end{smallmatrix}]\oplus
[\begin{smallmatrix} 3\\ 4\\ 5\end{smallmatrix}\to \begin{smallmatrix} &1\\ 2&& 3\end{smallmatrix}]\oplus
[\begin{smallmatrix} 5\\ 6\end{smallmatrix}\to \begin{smallmatrix} 4\\ 5\\ 6\end{smallmatrix}]\oplus
[\begin{smallmatrix} 5\\ 6\end{smallmatrix}\to 0]\oplus
[\begin{smallmatrix} 6\end{smallmatrix}\to 0]
\]
is a small projective generator. Studying the morphisms in the heart between these projective indecomposable complexes, it is easy to verify that $\Theta$ is the path algebra associated with the quiver
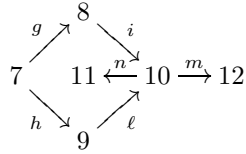
\[\xymatrix@-1pc{
&8\ar[rd]^i\\
7\ar[ru]^g\ar[rd]_h&11&10\ar[r]^m\ar[l]_n&12\\
&9\ar[ru]_\ell
}\]
with relations $mi=m\ell=0$.
\item
Let $\Lambda$ be the $k$-algebra given by the following quiver
\[\xymatrix@-1pc{
1\ar[r]&2\ar[r]&3\ar@(ur,dr)
}\]
Clearly, it is not an artinian algebra.
Consider the $\Lambda$-module $V=\begin{smallmatrix} 1\\ 2\end{smallmatrix}\oplus \begin{smallmatrix} 1\end{smallmatrix}$. The subcategory
$\Gen V$ is a torsion class; let us denote by $\Y$ the corresponding torsion free class. The ring $\Lambda_{V}:=\Lambda/\Ann_\Lambda V$ is the path algebra associated with the quiver $1\to 2$. It is easy to verify that the finitely presented $\Lambda$-module $V$ is a 
tilting $\Lambda_{V}$-module. Let us consider the following 
resolution of $V$:
\[
0
\to
\begin{smallmatrix} 3\\ 3\\ \vdots\end{smallmatrix}
\oplus
\begin{smallmatrix} 2 \\ 3\\ 3\\ \vdots\end{smallmatrix}
\to
\begin{smallmatrix} 1\\ 2 \\ 3\\ 3\\ \vdots\end{smallmatrix}
\oplus
\begin{smallmatrix} 1\\ 2 \\ 3\\ 3\\ \vdots\end{smallmatrix}
\to \begin{smallmatrix} 1\\ 2\end{smallmatrix}\oplus
\begin{smallmatrix} 1\end{smallmatrix}\to 0
\]
By Proposition~\ref{prop:projectivecover}, the complex with projective terms
\[\begin{smallmatrix} 3\\ 3\\ \vdots\end{smallmatrix}
\oplus
\begin{smallmatrix} 2 \\ 3\\ 3\\ \vdots\end{smallmatrix}
\to
\begin{smallmatrix} 1\\ 2 \\ 3\\ 3\\ \vdots\end{smallmatrix}
\oplus
\begin{smallmatrix} 1\\ 2 \\ 3\\ 3\\ \vdots\end{smallmatrix}\]
is a projective object in the heart. 
Let us consider the reject chain of $\Lambda$ with respect to the family of all quotients of $V$:
\[\Lambda=\begin{smallmatrix} 1\\ 2 \\ 3\\ 3\\ \vdots\end{smallmatrix}
\oplus
\begin{smallmatrix} 2 \\ 3\\ 3\\ \vdots\end{smallmatrix}
\oplus
\begin{smallmatrix} 3 \\ 3\\ 3\\ \vdots\end{smallmatrix}
\geq
\begin{smallmatrix} 3 \\ 3\\ 3\\ \vdots\end{smallmatrix}
\oplus
\begin{smallmatrix} 3 \\ 3\\ 3\\ \vdots\end{smallmatrix}
\oplus
\begin{smallmatrix} 3 \\ 3\\ 3\\ \vdots\end{smallmatrix}=...\]
The module $\begin{smallmatrix} 3 \\ 3\\ 3\\ \vdots\end{smallmatrix}
\oplus
\begin{smallmatrix} 3 \\ 3\\ 3\\ \vdots\end{smallmatrix}
\oplus
\begin{smallmatrix} 3 \\ 3\\ 3\\ \vdots\end{smallmatrix}$ is projective. Therefore by Proposition~\ref{prop:rejchain} the complex
\[\begin{smallmatrix} 3\\ 3\\ \vdots\end{smallmatrix}
\oplus
\begin{smallmatrix} 2 \\ 3\\ 3\\ \vdots\end{smallmatrix}
\oplus
\left(\begin{smallmatrix} 3\\ 3\\ \vdots\end{smallmatrix}\right)^3
\to
\begin{smallmatrix} 1\\ 2 \\ 3\\ 3\\ \vdots\end{smallmatrix}
\oplus
\begin{smallmatrix} 1\\ 2 \\ 3\\ 3\\ \vdots\end{smallmatrix}\]
is a small projective generator of the heart $\H(\Gen V, \Y)$; thus the latter is equivalent to a module category over a ring $\Theta$. 
Let us decompose our small projective generator as a direct sum of indecomposable projective complexes:
\[[\begin{smallmatrix} 3\\ 3\\ \vdots\end{smallmatrix}\to
\begin{smallmatrix} 1\\ 2 \\ 3\\ 3\\ \vdots\end{smallmatrix}]
\oplus
[\begin{smallmatrix} 2 \\ 3\\ 3\\ \vdots\end{smallmatrix}\to
\begin{smallmatrix} 1\\ 2 \\ 3\\ 3\\ \vdots\end{smallmatrix}]
\oplus
[\begin{smallmatrix} 3\\ 3\\ \vdots\end{smallmatrix}\to 0]^3
\]
Forgetting the redundant repetition we get that also 
\[[\begin{smallmatrix} 3\\ 3\\ \vdots\end{smallmatrix}\to
\begin{smallmatrix} 1\\ 2 \\ 3\\ 3\\ \vdots\end{smallmatrix}]
\oplus
[\begin{smallmatrix} 2 \\ 3\\ 3\\ \vdots\end{smallmatrix}\to
\begin{smallmatrix} 1\\ 2 \\ 3\\ 3\\ \vdots\end{smallmatrix}]
\oplus
[\begin{smallmatrix} 3\\ 3\\ \vdots\end{smallmatrix}\to 0]
\]
is a small projective generator. Studying the morphisms in the heart between these projective indecomposable complexes, it is easy to verify that $\Theta$ is the path algebra associated with the quiver
\[\xymatrix@-1pc{
4&5\ar[l]\ar[r]&6\ar@(ur,dr)
}\]
\end{enumerate}
\end{example}

An artin algebra $\Lambda$ is  \emph{quasi tilted} if there exists a faithful splitting torsion pair $(\X, \Y)$ such that any module in $\Y$ has projective dimension at most one (see for instance \cite{HRS}). In \cite{HR} it is  showed that any quasi tilted algebra $\Lambda$ of finite representation type is a tilted algebra, that is $\Lambda\cong \End_{\Gamma}(T)$, where $\Gamma$ is an hereditary algebra and $T$ a tilting $\Gamma$-module.  We get the same result applying Corollary~\ref{cor:artin}.

\begin{proposition}
If $\Lambda$ is a quasi tilted algebra of finite representation type, then $\Lambda$ is tilted.
\end{proposition}
\begin{proof}
Let $(\X,\Y)$ be a faithful splitting torsion pair in $\rMod \Lambda$ such that any module in $\Y$ has projective dimension at most one. By \cite{HRS} the heart $\H(\X, \Y)$ is an hereditary abelian category, and $\Lambda[1]$ is a tilting object in $\H(\X, \Y)$ with $\End_{\H(\X, \Y)}\Lambda[1]\cong \Lambda$. Therefore to get the thesis it is sufficient to prove that $\H(\X, \Y)$ is equivalent to a category of modules.
By  Corollary~\ref{cor:artin}, $\H(\X, \Y)$ is equivalent to a category of modules if and only if the torsion class $\X$ is generated by a  finitely presented $\Lambda$-module $P$ which is a tilting $\Lambda_P$-module.  

Let us consider the module  $Q={{\oplus_i}^n} X_i$ where $\{X_1, \dots , X_n\}$ is a complete list of non-isomorphic indecomposable modules in $\X$.  Since the algebra is of finite representation type, the torsion class $\X$ coincides with $\Add Q$; moreover, being $Q$ product complete, $\Add Q=\Prod Q$ is closed under products.

If  $\Ext^1(Q,Q)= 0$, we take $P:=Q$. Otherwise, if  $\Ext^1(Q,Q)\neq 0$,   let us denote by $X_i$ and $X_j$ two indecomposable summands of $Q$ such that $\Ext^1(X_j, X_i)\neq 0$. We claim that $\X=\Gen Q_1=\Pres Q_1$, where $Q_1=Q\setminus \{X_i\}$. Indeed, let  $0\to X_j\to M \to X_i\to 0$ be a non splitting exact sequence. Since in the valued quiver of $\Lambda$ there are no oriented cycles (see \cite{HRS}), we deduce that $X_i$ does not belong to $\add M$ and therefore $M$ belongs to $\add Q_1$. Since $M$ generates $X_i$, it is $\Gen Q=\Gen Q_1$ and $\Pres Q=\Pres Q_1$. If
$\Ext^1(Q_1, Q_1)=0$ we take $P:=Q_1$, otherwise we repeat the same procedure. In such a way,  in a finite number of steps, we will get a module $P:=Q_m$ with $\X=\Gen P=\Pres P$ and $\Ext^1(P,P)=0$.

The module $P$ is a finitely presented $\Lambda$-module; let us prove that it is a tilting $\Lambda_P$-module.  By Proposition~\ref{prop:quasiversustilting}, since $\X$ is closed under products, it is sufficient to prove that
$\Gen P=\overline\Gen P\cap P^\perp$.
If $M\in \overline\Gen P$, there exists an exact sequence $0\to M\to X_0 \to X_1\to 0$ with $X_0\in  \X$ and $\X_1\in \Add P$; if $M$ belongs also to $P^\perp$ this sequence splits and so $M$ belongs to $\Gen P$.  Conversely, if  $M$ belongs to  $\Gen P=\Pres P$, there exists en exact sequence $0\to M_0\to P_0\to M\to 0$, where $P_0\in \Add P=\Prod P$ and $M_0\in \Gen P$. Thus from the sequence $\Ext^1(P, P_0)\to \Ext^1(P, M)\to \Ext^2(P, M_0)$, since any module in $\X$ has injective dimension at most one (\cite{HRS}), we conclude that $M\in P^{\perp}$.
\end{proof}

\end{document}